\documentclass[12pt]{article}

\usepackage[a4paper,margin=1in,bottom=1.3in]{geometry} % bottom大きめで見切れ防止
\usepackage{amsthm, amsmath, amsfonts, amssymb, mathtools}
\usepackage{graphicx}
\usepackage[all]{xy}
\usepackage{color}
\usepackage{fancyhdr}

\newtheorem{Def}{Definition}[section]

\newtheorem{Thm}[Def]{Theorem}
\newtheorem{Lem}[Def]{Lemma}

\newtheorem{Ex}[Def]{Example}
\newtheorem{Fact}[Def]{Fact}

\newcommand{\Hom}{\mathrm{Hom}}

\newcommand{\SU}{\mathrm{SU}}

\makeatletter

\@addtoreset{equation}{section}
\makeatother

\makeatletter

\def\subjclass#1{\def\@subjclass{#1}}
\def\keywords#1{\def\@keywords{#1}}

\fancypagestyle{firstpagefooter}{
  \fancyhf{} % clear header/footer
  \fancyfoot[C]{%
    \footnotesize
    Affiliation: Graduate School of Advanced Science and Engineering, Hiroshima University\\
    Email: \texttt{ikedayuya@hiroshima-u.ac.jp}\\[4pt]
    Mathematics Subject Classification (2020): \@subjclass\\[2pt]
    Keywords: \@keywords
  }
}

\makeatother

\title{Designs on the Tautological bundle}
\author{Yuya Ikeda}
\date{} % arXiv では日付なし推奨


\begin{document}

\maketitle
\thispagestyle{firstpagefooter}

\begin{abstract}
In this paper, we introduce the framework of a generalized design, which represents any linear operator as a finite sum of local linear maps attached to finitely many points, thereby abstracting the core of design theory without employing integration.
We then construct such a design on the space of sections of the tautological bundle over the complex projective line. By using the irreducible decomposition of this space as an $\SU(2)$-representation, we show that the projection onto its lowest-dimensional summand can be realized as a finite sum of these local maps.
Our construction relies on invariant theory for the binary icosahedral group and an analysis of fixed-point subspaces in symmetric tensor representations.
\end{abstract}

\section{Introduction}

Design theory is a discipline that draws on combinatorics, geometry, harmonic analysis, and algebraic coding theory, and its origins can be traced back to Fisher’s classical work in statistics \cite{Fisher1935}.
In this foundational work, Fisher introduced principles such as randomization, blocking, and replication in order to impose a coherent structure on the arrangement of experimental units. These ideas provided a framework in which factor effects could be assessed accurately while controlling for errors and variance.

By abstracting this idea of structuring in order to enable valid inference into a purely combinatorial setting, Bose and his collaborators \cite{Bose1939,Bose1947,Bose1960} developed the theory of balanced incomplete block designs (BIBDs).
Given a finite set $X$ and a family $\mathcal{B}$ of its subsets, they required that for any two distinct points $x, y \in X$,
\[
\#\{\, B \in \mathcal{B} \mid x, y \in B \,\} = \lambda,
\]
meaning that the number of blocks containing the pair $\{x,y\}$ is constant.

In the 1960s, a close connection with coding theory was established, and the Assmus -- Mattson theorem \cite{AssmusMattson1969} showed that $t$-designs can be constructed from parameters of a linear code, in particular from its weight distribution.

Subsequently, Delsarte introduced the theory of association schemes \cite{Delsarte1973}, which laid the foundation for reformulating the combinatorial properties of finite designs in analytic and algebraic terms. This framework links designs to orthogonal decompositions of polynomial spaces and to linear programming techniques.

A significant development came with the introduction of spherical designs by Delsarte, Goethals, and Seidel \cite{DelsarteGoethalsSeidel1977}.
They defined a finite subset $X \subset S^{n-1}$ to be a spherical $t$-design if, for every polynomial $f$ of degree at most $t$, one has
\[
\frac{1}{|S^{n-1}|}\int_{S^{n-1}} f \, d\mu
  = \frac{1}{|X|} \sum_{x\in X} f(x).
\]
This shift in viewpoint extended design theory from the finite to the continuous setting.

Subsequently, the classification of tight spherical designs by Bannai and Damerell \cite{BannaiDamerell1979,BannaiDamerell1981}, the introduction of averaging sets by Seymour and Zaslavsky—which generalize spherical designs to arbitrary spaces and classes of functions \cite{SeymourZaslavsky1984} --- and further developments on designs in Grassmannian and complex projective spaces \cite{BachocCoulangeonNebe2002,RoyScott2007} were carried out. These advances were studied systematically within the framework of harmonic analysis and representation theory.

More recently, Bondarenko, Radchenko, and Viazovska proved that the number of points required to construct a spherical design can be bounded by a polynomial in the dimension \cite{Bondarenko2013}, yielding a decisive breakthrough in the existence problem for spherical designs.

However, if one returns to the historical origins of design theory, its essence lies in the idea that a finite set of points can faithfully represent a global geometric or analytic structure.

Although spherical and projective designs are typically formulated using integral operators, this reliance on integration arises solely from the existence of a natural invariant measure on the underlying space. It does not mean that the essential idea of design theory is intrinsically tied to integral structures.

The purpose of this paper is to return to the original conceptual foundation of design theory by introducing an abstract averaging framework which does not rely on integral operators.

That is, instead of assuming the existence of an invariant measure on the underlying space, we consider a setting consisting of a linear space $\mathcal{H}$, a family of linear spaces $\{ V_p \}_{p \in \Omega}$, and linear maps
\[
e_p : \mathcal{H} \to V_p
\]
for each $p \in \Omega$.
Given a finite subset $X \subset \Omega$, assigning a linear map
\[
\lambda_x : V_x \longrightarrow \mathcal{H}
\]
to each $x \in X$ allows one to express a linear operator $\tau$ on $\mathcal{H}$ as a finite sum of compositions of local linear maps:
\[
\tau(s) = \sum_{x \in X} (\lambda_x \circ e_x)(s).
\]
In this way, a linear operator can be reconstructed solely from local data attached to finitely many points. This perspective captures the essence of an abstract averaging structure that does not depend on integral operators.

Furthermore, to demonstrate that this abstract framework is effective in geometric and representation-theoretic settings, we construct a design on the space of sections $\Gamma(T_{2,1})$  of the tautological bundle $T_{2,1}$ over the complex projective line $\mathbb{C}P^{1}$. This serves as a concrete example illustrating the applicability of our approach.

In summary, this work
\begin{itemize}
\item systematizes the foundational idea of finite averaging in design theory as an abstract linear structure,
\item provides a general framework that encompasses the various notions developed in classical design theory, and
\item demonstrates its effectiveness through a concrete geometric example arising from the tautological bundle.
\end{itemize}
In doing so, it offers a new perspective that bridges discrete and analytic approaches within design theory.

\section{Definition of $\tau$-design}

Motivated by the fact that certain linear operators can be expressed as finite sums of local maps, we introduce the notion of a $\tau$ - design as a unified concept for finite averaging in an abstract setting.

\begin{Def}
\label{def-of-tau-design}
We set the notation as follows:
\begin{itemize}
\item $\Omega$: a set,
\item $W, \mathcal{H}$: complex linear spaces,
\item $\mathcal{H}_0$: a subspace of $\mathcal{H}$,
\item $\{ V_p \}_{p \in \Omega}$: a family of complex vector spaces,
\item $\{ e_p : \mathcal{H} \to V_p \}_{p \in \Omega}$: a collection of linear maps,
\item $\tau : \mathcal{H}_0 \to W$: a linear map.
\end{itemize}
A pair $(X, \{ \lambda_x \}_{x \in X})$, where $X \subset \Omega$ is finite and each
$\lambda_x : V_x \to W$ is a linear map, is called a $\tau$-design if it satisfies the identity
\[
\tau(s) = \sum_{x \in X} (\lambda_x \circ e_x)(s) \quad (s \in \mathcal{H}_0).
\]
\end{Def}

\begin{Ex}
\label{spherical-tau}
We observe that spherical $t$-designs can be reinterpreted as a special case of $\tau$-designs.

Set the data as follows:
\begin{itemize}
  \item $\Omega = S^{n-1}$,
  \item $V = \mathbb{R}$,\quad $H = C^\infty(\mathbb{R}^n)$,\quad $H_0 = \mathrm{Pol}_{\le t}(\mathbb{R}^n)\vert_{S^{n-1}}$,
  \item $V_p = \mathbb{R}$ for each $p \in \Omega$,
  \item $e_p(f) = f(p)$,
  \item $\tau(f) = \frac{1}{|S^{n-1}|} \int_{S^{n-1}} f \, d\mu$ where $\mu$ is the rotation-invariant Radon measure,
  \item $\lambda_x(z) = \frac{1}{|X|} z$.
\end{itemize}

With these choices,
\[
X \text{ is a spherical $t$-design}
\quad\Longleftrightarrow\quad
(X,\lambda) \text{ is a $\tau$-design}.
\]
This is immediate from the definitions.
\end{Ex}

\section{Definition of Tautological bundle}

\begin{Def}
\label{taurological}
Consider the complex projective line $\mathbb{C}P^{1}$ and define
\begin{equation}
T_{2,1} = \{\, (V, v) \in \mathbb{C}P^{1} \times \mathbb{C}^{2} \mid v \in V \,\}.
\end{equation}
We define an action of $\SU(2)$ by
\begin{align*}
g \cdot V &= \{\, gv \mid v \in V \,\}, \\
g \cdot (V, v) &= (gV, gv).
\end{align*}
Let $\pi : T_{2,1} \to \mathbb{C}P^{1}$ be the projection defined by
\[
\pi(V, v) = V.
\]
Then $\pi$ is an $\SU(2)$-equivariant line bundle, and $T_{2,1}$ is called the tautological bundle.
\end{Def}

In particular, as shown by Koga--Nagatomo \cite[p.~499]{KogaNagatomo}, the tautological bundle $T_{2,1}$ admits the $\SU(2)$-equivariant description 
\[
T_{2,1} \cong \SU(2) \times_{U(1)} \mathbb{C}_{1},
\]
where $\mathbb{C}_1$ denotes the one-dimensional $U(1)$-representation of weight $1$.

\section{Irreducible decomposition of $\Gamma(T_{2,1})$}

In this section, we study the $\SU(2)$-equivariant line bundle $T_{2,1}$ over the complex projective line $\mathbb{CP}^1 \cong \SU(2)/U(1)$.  
We then analyze the decomposition of its space of sections $\Gamma(T_{2,1})$ into irreducible$\SU(2)$-representations.
Throughout, following Nagatomo~\cite{Nagatomo2025}, we let $\Gamma(E)$ denote the $L^{2}$-space of sections of a vector bundle $E$.

The irreducible representations of $U(1)$ are precisely the one-dimensional representations
indexed by integer weights $\lambda \in \mathbb{Z}$,
\[
  \mathbb{C}_{\lambda} : e^{i\theta} \mapsto e^{i\lambda\theta}.
\]

As stated in Nagatomo \cite[Example]{Nagatomo2025}, one has
\[
\mathbb{CP}^1 = \SU(2)/U(1),
\]
and the line bundle of degree $1$,
\[
\mathcal{O}(1) \to \mathbb{CP}^1,
\]
is given as the $\SU(2)$-equivariant line bundle
\[
\mathcal{O}(1) \cong \SU(2) \times_{U(1)} \mathbb{C}_{-1}.
\]
In general, the dual of an equivariant line bundle satisfies
\[
  \bigl(\SU(2)\times_{U(1)} \mathbb{C}_{\lambda}\bigr)^{\vee}
  \cong \SU(2)\times_{U(1)} \mathbb{C}_{-\lambda},
\]
and hence
\[
  T_{2,1} = O(1)^{\vee}
  \cong \SU(2)\times_{U(1)} \mathbb{C}_{1}.
\]

As explained by Widdows \cite[§1.2]{Widdows2000}, if $V_1 \cong \mathbb{C}^2$ denotes the
fundamental representation of $\SU(2) \cong Sp(1)$, then its $n$-th symmetric power
\[
V_n := S^n(V_1)
\]
is an irreducible representation of $\SU(2)$.  Moreover, every finite-dimensional irreducible representation of $\SU(2)$ is isomorphic to $S^n(V_1)$ for some $n \ge 0$.

The same reference shows that $V_n$ decomposes into $U(1)$-weight spaces as
\[
V_n = \bigoplus_{j=0}^{n} \mathbb{C}_{\,n-2j},
\]
and each weight occurs with multiplicity one.

Nagatomo~\cite[Eq.~(7.2)]{Nagatomo2025} gives the $U(1)$-weight decomposition of the symmetric power
$S_{k}\mathbb{C}^{2}$ as
\[
  S_{k}\mathbb{C}^{2}
  = \mathbb{C}_{k} \oplus \mathbb{C}_{k-2} \oplus \cdots \oplus \mathbb{C}_{-k}.
\]
The same paper also states that the line bundle $O(k)$ is the equivariant line bundle
\[
  O(k) \cong \SU(2) \times_{U(1)} \mathbb{C}_{-k},
\]
and that its $L^{2}$-space of sections decomposes as an $\SU(2)$-representation by
\[
  \Gamma(O(k))
  = \sum_{l=0}^{\infty} S_{|k|+2l}\mathbb{C}^{2}
  \qquad\text{\cite[Eq.~(7.3)]{Nagatomo2025}}.
\]

Since $T_{2,1} = O(1)^{\vee} \cong O(-1)$, substituting $k = -1$ into
Nagatomo~\cite[Eq.~(7.3)]{Nagatomo2025} yields
\[
  \Gamma(T_{2,1})
  \cong
  \Gamma(O(-1))
  = \bigoplus_{k \ge 1} S_{2k-1}\mathbb{C}^{2},
\]
as an $\SU(2)$-representation.

Furthermore, by Widdows~\cite[§1.2]{Widdows2000}, the symmetric power
$S^{n}(\mathbb{C}^{2})$ is the $(n+1)$-dimensional irreducible
representation of $\SU(2)$.  Hence each summand satisfies
\[
  \dim_{\mathbb{C}} \Gamma^{2k}(T_{2,1})
  = \dim_{\mathbb{C}} S^{2k-1}(\mathbb{C}^{2})
  = 2k.
\]

From the above discussion, we obtain the following fact.

\begin{Fact}
\label{irr}
As an $\SU(2)$-representation, the space $\Gamma(T_{2,1})$ admits the irreducible decomposition
\[
  \Gamma(T_{2,1})
  = \bigoplus_{k \ge 1} \Gamma^{2k}(T_{2,1}),
\]
where each summand satisfies $\dim_{\mathbb{C}}\Gamma^{2k}(T_{2,1}) = 2k$, and moreover
$\Gamma^{2k}(T_{2,1})$ is isomorphic, as an $\SU(2)$-representation, to the symmetric tensor space
\(
  S^{2k-1}(\mathbb{C}^{2}).
\)
\end{Fact}

\section{Hilbert series and rings of invariant polynomials}

Let $\mathbb{C}[x_{1}, x_{2}]$ denote the polynomial ring in two complex variables, and let $\mathbb{C}[x_{1}, x_{2}]_{d}$ be its homogeneous component of degree $d$.
For a subgroup $G \subset GL(2,\mathbb{C})$, the ring of invariant polynomials is defined by
\[
  \mathbb{C}[x_{1}, x_{2}]^{G}
  = \{\, f \in \mathbb{C}[x_{1}, x_{2}] \mid f(Ax) = f(x) \ \text{for all } A \in G \,\},
\]
and we set
\[
  \mathbb{C}[x_{1}, x_{2}]^{G}_{d}
  := \mathbb{C}[x_{1}, x_{2}]_{d} \cap \mathbb{C}[x_{1}, x_{2}]^{G}.
\]
The polynomial
\[
  P^{G}_{2}(t)
  := \sum_{d \ge 0} \dim_{\mathbb{C}}\bigl(\mathbb{C}[x_{1}, x_{2}]^{G}_{d}\bigr)\, t^{d}
\]
is called the Hilbert series of the invariant polynomial ring for $G$.
For a finite group $G$, Molien's formula expresses the Hilbert series as
\[
P^{G}_2(t)
 = \frac{1}{|G|} \sum_{A \in G} \frac{1}{\det(I - tA)},
\]
see \cite[Theorem~1.10]{Mukai1996}.

In what follows, we explicitly describe $P_{2}^{G}(t)$ for the binary icosahedral group,
a classical example in invariant theory.  Let $\varepsilon$ be a primitive fifth root of unity, and set
\[
S =
\begin{pmatrix}
  \varepsilon^{3} & 0 \\[2mm]
  0 & \varepsilon^{2}
\end{pmatrix},
\qquad
T =
\frac{1}{\varepsilon^{2} - \varepsilon^{3}}
\begin{pmatrix}
  \varepsilon + \varepsilon^{4} & 1 \\
  1 & -\varepsilon - \varepsilon^{4}
\end{pmatrix},
\qquad
U =
\begin{pmatrix}
  0 & 1 \\
  -1 & 0
\end{pmatrix}.
\]

\begin{Fact}[{\cite[Example~1.16]{Mukai1996}}]
\label{icosa}
The subgroup of $\SU(2)$ generated by these elements,
\[
  G_{\mathrm{icosa}} := \langle S, T, U \rangle \subset \SU(2),
\]
is called the binary icosahedral group, and it has order $120$.
The Hilbert series of the invariant ring $\mathbb{C}[x_{1}, x_{2}]^{G_{\mathrm{icosa}}}$ is
\begin{align*}
  P_{2}^{G_{\mathrm{icosa}}}(t)
   &= \frac{1 + t^{30}}{(1 - t^{12})(1 - t^{20})} \\[1mm]
   &= 1 + t^{12} + t^{20} + t^{24} + \cdots .
\end{align*}
In particular, there are no nonzero $G_{\mathrm{icosa}}$-invariant homogeneous polynomials of degree
$1 \le d \le 11$.
\end{Fact}

From the above fact, together with the observation that the $d$-th symmetric tensor space
$S^{d}(\mathbb{C}^{2})$ is isomorphic to $\mathbb{C}[x_{1}, x_{2}]_{d}$ as a $G$-representation, we obtain
\[
\dim_{\mathbb{C}}\bigl(S^{d}(\mathbb{C}^{2})\bigr)^{G_{\mathrm{icosa}}}
=
\begin{cases}
  1 & (d = 0), \\[2mm]
  0 & (1 \le d \le 11).
\end{cases}
\]

\section{Designs for the Projection on $\Gamma(T_{2,1})$}

Let
\[
H = \bigoplus_{k=1}^{5} \Gamma^{2k}(T_{2,1}),
\]
and define the projection
\[
\tau : H \to \Gamma^{2}(T_{2,1}).
\]

Set
\[
\Omega = \mathbb{CP}^1,\quad
H = \Gamma(T_{2,1}),\quad
H_0 = H.
\]
For $l \in \mathbb{CP}^1$, we define a map
\[
e_l : \Gamma(T_{2,1}) \to \mathbb{C}, \quad
e_l(s) = s^{(2)} (l),
\]
where $s^{(2)}(l)$ denotes the second component of the vector $s(l) \in l \subset \mathbb{C}^2$.

For a subset $X \subset \mathbb{C}P^{1}$ and a collection of linear maps
$\{\lambda_x : V_x \to \Gamma^{2}(T_{2,1})\}_{x \in X}$, we define
\[
  \Psi_{(X,\lambda)}(s)
  = \sum_{x \in X} \lambda_{x} \circ e_{x}(s).
\]

For $g \in G \subset \SU(2)$ and $x \in X$, we define
\[
(g \cdot \lambda_x)(z) = g \cdot \lambda_x(g^{-1} z),
\]
and extend this to
\[
(g \cdot \Psi_{(X,\lambda)})(s)
   = \sum_{x \in X} (g \cdot \lambda_x) \circ e_{g x}(s).
\]

Furthermore, for two pairs we define
\[
  (X,\lambda_{X}) + (Y,\lambda_{Y})
  = (X \cup Y,\, \lambda_{X \cup Y}),
\]
where
\[
  \lambda_{X \cup Y}(z)
  =
  \begin{cases}
    \lambda_{X}(z) & (z \in X \setminus Y), \\[1mm]
    \lambda_{X}(z) + \lambda_{Y}(z) & (z \in X \cap Y), \\[1mm]
    \lambda_{Y}(z) & (z \in Y \setminus X).
  \end{cases}
\]

Finally, we define the $G$ - orbit sum of $(X,\lambda)$ by
\[
  G \cdot (X,\lambda)
  := \sum_{g\in G} (gX,\, g\lambda).
\]

\begin{Lem}
\label{e-lem}
For any $g,h \in \SU(2)$, $x \in \mathbb{CP}^1$, and $s \in H$, one has
\[
e_{h(gx)}(h \cdot s) = h \cdot e_{g x}(s).
\]
\end{Lem}

\begin{proof}
Take arbitrary $g, h \in \SU(2)$, $x \in \mathbb{C}P^{1}$, and $s \in H$.
For $l = h(gx)$, the value $e_{h(gx)}(h \cdot s)$ is the second component of $(h \cdot s)(l)$, and we have
\begin{align*}
(h \cdot s)(l)
  &= h \cdot \bigl(s(h^{-1} l)\bigr) \\
  &= h \cdot \bigl(s(gx)\bigr).
\end{align*}
Hence, by the definition of the $\SU(2)$ - action on $T_{2,1}$, we obtain
\[
  e_{h(gx)}(h \cdot s) = h \cdot e_{gx}(s).
\]
\end{proof}

\begin{Lem}
\label{intertwining}
Let $G \subset \SU(2)$ be a finite group, and set
\[
(Y,\lambda_Y) := \frac{1}{|G|} \, G \cdot (X,\lambda_X).
\]
Then $\Psi_{(Y,\lambda_Y)}$ is $G$-equivariant.
\end{Lem}

\begin{proof}
Take any $h \in G$ and $s \in H$. It suffices to show that
\[
  \Psi_{(Y,\lambda_{Y})}(h \cdot s)
  = h \cdot \Psi_{(Y,\lambda_{Y})}(s).
\]
We compute:
\begin{align*}
(\text{left side})
 &= \frac{1}{\# G}
    \sum_{g \in G} \sum_{x \in X}
      (g \cdot \lambda_{x}) \circ e_{g x}(h \cdot s) \\[1mm]
 &= \frac{1}{\# G}
    \sum_{g \in G} \sum_{x \in X}
      (g \cdot \lambda_{x})\bigl(h \cdot e_{h^{-1} g x}(s)\bigr) \\[1mm]
 &= \frac{1}{\# G}
    \sum_{g \in G} \sum_{x \in X}
      g \cdot \lambda_{x}\bigl(g^{-1} h \cdot e_{h^{-1} g x}(s)\bigr) \\[1mm]
 &= \frac{1}{\# G}
    \sum_{g \in G} \sum_{x \in X}
      h\,(h^{-1} g) \cdot \lambda_{x}
      \bigl((h^{-1} g)^{-1} \cdot e_{h^{-1} g x}(s)\bigr) \\[1mm]
 &= \frac{1}{\# G}
    \sum_{g' \in G} \sum_{x \in X}
      h\,(g' \cdot \lambda_{x}) \circ e_{g' x}(s) \\[1mm]
 &= (\text{right side}).
\end{align*}
\end{proof}

The following theorem is the main theorem of this paper.

\begin{Thm}
\label{design-thm}
Let $X \subset \mathbb{C}P^{1}$ be a subset.  
If $\operatorname{tr}_{\mathbb{C}}(\Psi_{(X,\lambda)}) = 2$, then
\[
  (Y,\lambda_{Y})
  = \frac{1}{\# G_{\mathrm{icosa}}}\,
    G_{\mathrm{icosa}} \cdot (X,\lambda)
\]
is a $\tau$-design.

In other words, the operator $\Psi_{(Y,\lambda)}$ satisfies
\[
  \Psi_{(Y,\lambda)}\big|_{\Gamma^{2k}(T_{2,1})}
  =
  \begin{cases}
    \mathrm{id} & (k = 1), \\[2mm]
    0 & (k = 2,3,4,5).
  \end{cases}
\]
\end{Thm}

\begin{proof}
Write $G = G_{\mathrm{icosa}}$.  
By the preceding lemma, the map
\[
  \Psi_{(Y,\lambda_{Y})} : H \to \Gamma^{2}(T_{2,1})
\]
is $G$-equivariant.  
Hence, its restriction to each irreducible summand
\[
  \Psi_{k}
  := \Psi_{(Y,\lambda_{Y})}\big|_{\Gamma^{2k}(T_{2,1})}
\]
satisfies
\[
  \Psi_{k}
  \in \Hom_{G}\!\bigl(\Gamma^{2k}(T_{2,1}),\,\Gamma^{2}(T_{2,1})\bigr).
\]

For finite group representations, one has the canonical isomorphism
\[
  \Hom_{G}(V,W)
  \cong (W \otimes V^{\ast})^{G},
\]
and therefore
\[
  \Hom_{G}\!\bigl(\Gamma^{2k}(T_{2,1}),\,\Gamma^{2}(T_{2,1})\bigr)
  \cong
  \bigl(\Gamma^{2}(T_{2,1}) \otimes (\Gamma^{2k}(T_{2,1}))^{\ast}\bigr)^{G}.
\]

Using the isomorphism
\[
  \Gamma^{2k}(T_{2,1})
  \cong S^{2k-1}(\mathbb{C}^{2})
\]
as $\SU(2)$-representations (Fact~\ref{irr}), together with the self-duality of
$S^{m}(\mathbb{C}^{2})$ as an $\SU(2)$-representation, we obtain
\[
  \Gamma^{2}(T_{2,1}) \otimes (\Gamma^{2k}(T_{2,1}))^{\ast}
  \cong
  S^{1}(\mathbb{C}^{2}) \otimes S^{2k-1}(\mathbb{C}^{2}).
\]

Furthermore, by the Clebsch--Gordan decomposition
\cite[Appendix~C]{Hall2015},
\[
  S^{1}(\mathbb{C}^{2}) \otimes S^{2k-1}(\mathbb{C}^{2})
  \cong
  S^{2k}(\mathbb{C}^{2}) \oplus S^{2k-2}(\mathbb{C}^{2}),
\]
it follows that
\[
  \Hom_{G}\!\bigl(\Gamma^{2k}(T_{2,1}),\,\Gamma^{2}(T_{2,1})\bigr)
  \cong
  \bigl(S^{2k}(\mathbb{C}^{2}) \oplus S^{2k-2}(\mathbb{C}^{2})\bigr)^{G}.
\]

Therefore,
\[
  \dim_{\mathbb{C}}
  \Hom_{G}\!\bigl(\Gamma^{2k}(T_{2,1}),\,\Gamma^{2}(T_{2,1})\bigr)
  =
  \dim_{\mathbb{C}} \bigl(S^{2k}(\mathbb{C}^{2})\bigr)^{G}
  \;+\;
  \dim_{\mathbb{C}} \bigl(S^{2k-2}(\mathbb{C}^{2})\bigr)^{G}.
\]

Hence,
\[
  \dim_{\mathbb{C}}
  \Hom_{G}\!\bigl(\Gamma^{2k}(T_{2,1}),\,\Gamma^{2}(T_{2,1})\bigr)
  =
  \begin{cases}
    1 & (k = 1), \\[2mm]
    0 & (k = 2,3,4,5).
  \end{cases}
\]

From the above, we obtain
\[
  \Hom_{G}\!\bigl(\Gamma^{2k}(T_{2,1}),\,\Gamma^{2}(T_{2,1})\bigr)
  =
  \begin{cases}
    \mathbb{C} & (k = 1), \\[2mm]
    0 & (k = 2,3,4,5),
  \end{cases}
\]
and therefore each restriction of $\Psi_{(Y,\lambda_{Y})}$,
\[
  \Psi_{k}
   := \Psi_{(Y,\lambda_{Y})}\big|_{\Gamma^{2k}(T_{2,1})},
\]
is of the form
\[
  \Psi_{k}
  =
  \begin{cases}
    c\,\mathrm{id}_{\Gamma^{2}(T_{2,1})} & (k = 1), \\[2mm]
    0 & (k = 2,3,4,5),
  \end{cases}
\]
for some $c \in \mathbb{C}$.
Hence,
\[
  \Psi_{(Y,\lambda_{Y})}
  = c\,\mathrm{id}_{\Gamma^{2}(T_{2,1})}.
\]
It follows that
\[
  \operatorname{tr}_{\mathbb{C}}\!\bigl(\Psi_{(Y,\lambda_{Y})}\bigr)
   = \operatorname{tr}_{\mathbb{C}}\!\bigl(c\,\mathrm{id}_{\Gamma^{2}(T_{2,1})}\bigr)
   = 2c.
\]

On the other hand, we have
\begin{align*}
(g \cdot \Psi_{(X,\lambda)})(s)
  &= \sum_{x \in X} (g \cdot \lambda_{x}) \circ e_{gx}(s) \\
  &= \sum_{x \in X} g \cdot \lambda_{x}\bigl(g^{-1} e_{gx}(s)\bigr) \\
  &= \sum_{x \in X} g \cdot \lambda_{x}\bigl(e_{x}(g^{-1} \cdot s)\bigr)
     = (g \circ \Psi_{(X,\lambda)} \circ g^{-1})(s),
\end{align*}
and therefore
\[
  g \cdot \Psi_{(X,\lambda)}
  = g \circ \Psi_{(X,\lambda)} \circ g^{-1}.
\]

Thus, by conjugation invariance of the trace and its linearity,
\[
  \operatorname{tr}_{\mathbb{C}}\!\bigl(\Psi_{(Y,\lambda_{Y})}\bigr)
  = \operatorname{tr}_{\mathbb{C}}\!\bigl(\Psi_{(X,\lambda)}\bigr).
\]
By assumption, $\operatorname{tr}_{\mathbb{C}}\!\bigl(\Psi_{(X,\lambda)}\bigr) = 2$, and hence $c = 1$.

Therefore,
\[
  \Psi_{(Y,\lambda_{Y})}\big|_{\Gamma^{2k}(T_{2,1})}
  =
  \begin{cases}
    \mathrm{id} & (k = 1), \\[2mm]
    0 & (k = 2,3,4,5).
  \end{cases}
\]

\end{proof}

As a concrete application of Theorem\ref{design-thm}, we now construct an explicit $\tau$-design.

Let
\[
  x_{0} = \{\, (x,0) \mid x \in \mathbb{C} \,\},
\]
and define
\[
  \lambda_{0} : x_{0} \to \Gamma^{2}(T_{2,1}), \qquad
  \lambda_{0}(x,0)(l) = (\,l,\; 2\,\mathrm{pr}_{l}(x,0)\,),
\]
where $\mathrm{pr}_{l}(x,0)$ denotes the projection of $(x,0)$ onto the fiber over $l$.

We compute
\[
  \operatorname{tr}_{\mathbb{C}}\!\bigl(\Psi_{(x_{0},\lambda_{0})}\bigr)
  = \operatorname{tr}_{\mathbb{C}}\!\bigl(\lambda_{0} \circ e_{x_{0}}\bigr)
  = \operatorname{tr}_{\mathbb{C}}\!\bigl(e_{x_{0}} \circ \lambda_{0}\bigr).
\]
For any $(z,0) \in x_{0}$, we have
\begin{align*}
  \lambda_{0}(z,0)(x_{0})
    &= (\,x_{0},\, 2\,\mathrm{pr}_{x_{0}}(z,0)\,) \\
    &= (\,x_{0},\, 2z\,),
\end{align*}
and hence
\[
  (e_{x_{0}} \circ \lambda_{0})(z) = 2z.
\]
Therefore,
\[
  \operatorname{tr}_{\mathbb{C}}\!\bigl(\Psi_{(x_{0},\lambda_{0})}\bigr) = 2.
\]

By Theorem~\ref{design-thm}, we conclude that
\[
  \frac{1}{\# G_{\mathrm{icosa}}}\,
    G_{\mathrm{icosa}} \cdot (x_{0},\lambda_{0})
\]
is a $\tau$-design.  
In other words,
\begin{equation}
\label{pj-design}
  \tau
  = \frac{1}{120}
    \sum_{g \in G_{\mathrm{icosa}}}
      (g \cdot \lambda_{0}) \circ e_{g x_{0}}.
\end{equation}

Here we consider the stabilizer of $x_0$
\[
  G_{x_0}
  := \mathrm{Iso}_{x_0}(G_{\mathrm{icosa}})
  = \{ g \in G_{\mathrm{icosa}} \mid g \cdot x_0 = x_0 \},
\]
for which the following lemma holds.

\begin{Lem}
\label{stab-lemma}
For every $k \in G_{x_0}$, one has
\[
  k \cdot \lambda_0 = \lambda_0.
\]
\end{Lem}

\begin{proof}
Take arbitrary $k \in G_{x_0}$, $(z,0)\in x_0$, and $l \in \mathbb{C}P^{1}$.
By the definition of the $G$-action on $\lambda_x$ and on sections,
together with the definition of $\lambda_0$, we compute
\begin{align*}
  ((k \cdot \lambda_0)(z,0))(l)
  &= k \cdot \bigl( \lambda_0(k^{-1}(z,0))(k^{-1}l) \bigr) \\
  &= k \cdot \bigl( k^{-1}l,\; 2\,\mathrm{pr}_{k^{-1}l}(k^{-1}(z,0)) \bigr) \\
  &= \bigl( l,\; 2\,k\,\mathrm{pr}_{k^{-1}l}(k^{-1}(z,0)) \bigr).
\end{align*}
Since $k \in \SU(2)$ is unitary, orthogonal projection is $\SU(2)$-equivariant:
\[
  k \cdot \mathrm{pr}_{k^{-1}l}(k^{-1}(z,0))
  = \mathrm{pr}_{l}(z,0).
\]
Therefore
\[
  ((k \cdot \lambda_0)(z,0))(l)
  = (l,\, 2\,\mathrm{pr}_{l}(z,0))
  = \lambda_0(z,0)(l),
\]
which proves the claim.
\end{proof}

By this lemma, equation~\eqref{pj-design} can be rewritten more explicitly
in terms of orbit representatives.

Now set $G = G_{\mathrm{icosa}}$ and $G_{x_0} = \mathrm{Iso}_{x_0}(G_{\mathrm{icosa}})$.
Since $\#G = 120$ and $\#G_{x_0} = 10$, the left coset space $G / G_{x_0}$
admits a complete set of representatives $\{ g_1, \dots, g_{12} \}$.
Define
\[
  x_i = g_i x_0 \in \mathbb{C}P^{1}, \qquad
  \lambda_i = g_i \cdot \lambda_0 : x_i \to \Gamma(T_{2,1}).
\]
For $h \in G_{x_0}$ one has $h x_0 = x_0$, and by Lemma~\ref{stab-lemma},
$h \cdot \lambda_0 = \lambda_0$. Hence, for any $s \in \Gamma(T_{2,1})$,
\begin{align*}
  \tau(s)
  &= \frac{1}{120}\sum_{g\in G} (g \cdot \lambda_0)\circ e_{g x_0}(s) \\
  &= \frac{1}{120}\sum_{i=1}^{12}\sum_{h\in G_{x_0}}
        (g_i h \cdot \lambda_0)\circ e_{g_i h x_0}(s) \\
  &= \frac{1}{120}\sum_{i=1}^{12}\sum_{h\in G_{x_0}}
        (g_i \cdot \lambda_0)\circ e_{g_i x_0}(s) \\
  &= \frac{1}{120}\sum_{i=1}^{12}\sum_{h\in G_{x_0}}
        \lambda_i \circ e_{x_i}(s) \\
  &= \frac{1}{12}\sum_{i=1}^{12} \lambda_i \circ e_{x_i}(s).
\end{align*}

\end{document}